\documentclass[12pt,reqno]{amsart}

\usepackage{amsbsy,amsfonts,amsmath,amssymb,amscd,amsthm,mathrsfs,verbatim,calc}
\usepackage{graphicx,epsfig,color}
\usepackage{enumitem}
\usepackage[a4paper,text={6.1in,8in},centering,includefoot,foot=1in]{geometry}
\usepackage[colorlinks=true, linkcolor=blue, citecolor=blue, urlcolor=blue, backref]{hyperref}

\bibliography{yourbibfile}

\small\normalsize
\theoremstyle{plain}

\newtheorem{theorem}{Theorem}[section]
\newtheorem{lemma}[theorem]{Lemma}

\newtheorem{proposition}[theorem]{Proposition}
\newtheorem{corollary}[theorem]{Corollary}
\newtheorem{definition}[theorem]{Definition}
\newtheorem{definitions}[theorem]{Definitions}

\theoremstyle{definition}

\newtheorem{remarks}[theorem]{Remarks}
\newtheorem{example}[theorem]{Example}

\numberwithin{equation}{section}
\let\oldmarginpar\marginpar
\renewcommand\marginpar[1]{\-\oldmarginpar[\raggedleft\footnotesize \textcolor{red}{#1}]{\raggedright\footnotesize\textcolor{red}{#1}}}

\begin{document}
\title[Generalizations of interval and proper interval graphs]{Generalizations of interval and proper interval graphs for simplicial complexes}

\author[F. Khosh-Ahang]{Fahimeh Khosh-Ahang Ghasr}
\address{{Department of Mathematics, Ilam University,
P.O.Box 69315-516, Ilam, Iran.}}
\email{f.khoshahang@ilam.ac.ir and fahime$_{-}$khosh@yahoo.com}

\begin{abstract}
We introduce and investigate generalizations of interval and proper interval graphs to simplicial complexes, including strong interval, unit interval, and under closed variants. Through equivalent combinatorial and algebraic characterizations, we uncover hierarchies among these classes and extend key results to higher dimensions, such as the equivalence of closed and proper interval graphs. These formulations enable significant applications, including finding conditions for the sortability of $d$-independence complexes, constructions of normal Cohen-Macaulay domains linked to $d$-unit interval graphs, and forbidden subgraph theorems establishing chordality and $d$-claw-freeness. Our work advances the connections between graph theory, simplicial complexes, and commutative algebra, offering new insights into the algebraic underpinnings of combinatorial structures.
\end{abstract}

\subjclass[2010]{ 05C40, 05C75, 05E45, 13H10}

%	  05E40   Combinatorial aspects of commutative algebra
%     05E45   Combinatorial aspects of simplicial complexes
%     05C75   Structural characterization of families of graphs
%     05C25   graphs and abstract algebra (groups, rings, ...)
%     05C50   graphs and linear algebra
%     05C40   connectivity

%	  06A11   Algebraic aspects of posets

%     11C20   Matrices , determinants in number theory
%     11C08   polynomials and matrices

%     13P10   Grobner basis other bases for ideals and modules
%     13C14   Cohen-macaulay modules
%	  13H10   Special types (Cohen-Macaulay, Gorenstein, Buchsbaum, etc.)
%	  13D02   Syzygies, resolutions, complexes
%	  13A02   Graded rings
%  	  13F20   Polynomial rings and ideals; rings of integer-valued polynomials
%	  13A18   Valuations and their generalizations
%     13C40   linkage, complete intersection and determinantal ideal

%	  14M25   Toric varieties, Newton polyhedra [See also 52B20]

%	  16S36   Ordinary and skew polynomial rings and semigroup rings

\keywords{interval graph, proper interval graph, strong interval simplicial complex, under closed simplicial complex, unit interval simplicial complex.}

\maketitle
\setcounter{tocdepth}{1}
%\tableofcontents

%<<<<<<<<<<<<<<<<<<<<<<<<<<<<<<<<<<<<<<<<<<<<<<
\section{Introduction and Preliminaries}
In this paper, we explore connections between graph theory, simplicial complexes, and commutative algebra, focusing on generalizations of interval graphs and their higher-dimensional independence complexes. These structures arise naturally in various applications, including graph algorithms, bioinformatics, and algebraic combinatorics, and have been extensively studied over the past several decades \cite{B+S+V, Herzog2, Gardi, Herzog1, Fahimeh, R2}.

We begin with some notation and definitions. Throughout, $G = (V(G), E(G))$ denotes a simple graph on $n$ vertices, $d \in \mathbb{N}$, and $\mathbb{K}$ is a field. Let $\Delta$ be a pure $d$-dimensional simplicial complex (briefly, a $d$-complex) on $n$ vertices, with facet set denoted by $\mathcal{F}(\Delta)$.  We write any subset $F=\{i_1, \dots , i_r\}$ of $[n]$ with $1\leq i_1 < \dots < i_r \leq n$ as $F: i_1 \dots i_r$.   For an integer $0\leq k \leq d$, $k$-skeleton of $\Delta$ is the simplicial complex whose facets are all $k$-dimensional faces of $\Delta$. $\Delta$ is called connected if for every two vertices $u$ and $v$ there is a sequence $r_0, G_1, r_1, \dots , r_{s-1}, G_s, r_s$, with $u=r_0$ and $v=r_s$ such that for each $k$ with $1\leq k \leq s$, $r_k$ is a vertex of $\Delta$ and $G_k$ is a facet containing $r_{k-1}$ and $r_k$. For a subset $U = \{u_1, \dots, u_k\}$ of $V(G)$, the induced subgraph of $G$ on $U$ is denoted by $G[U]$ or $G[u_1, \dots, u_k]$. A \textbf{clique} in $G$ is the vertex set of any complete subgraph, and a \textbf{maximal clique} is one that is maximal under inclusion. A \textbf{noncut vertex} of $G$ is a vertex whose removal doesn't effect the number of connected components of $G$. The path and cycle graphs on $n$ vertices are denoted by $P_n$ and $C_n$, respectively. An \textbf{induced cycle} of length $k$ in $G$ is formed by distinct vertices $x_1, \dots, x_k$ such that $G[x_1, \dots, x_k] \cong C_k$; we denote this cycle by $(x_1, \dots, x_k)$ if $x_1 x_k \in E(G)$ and $x_i x_{i+1} \in E(G)$ for $1 \leq i \leq k-1$. A subset $U $ of $V(G)$ is a \textbf{$d$-independent set} if every connected component of $G[U]$ has at most $d$ vertices. The union of graphs $G$ and $H$ (with possibly overlapping vertices) is denoted by $G \cup H$. 

We associate to $G$ the simplicial complex $\Delta_d(G)$ on vertex set $V(G)$, whose facets are all $(d+1)$-element subsets $U $ of $V(G)$ such that $G[U]$ is connected. For $d=1$, the facet ideal of $\Delta_d(G)$ coincides with the edge ideal of $G$.
The \textbf{independence complex} $\mathrm{Ind}(G)$ of $G$ is the simplicial complex whose faces are the independent sets of $G$ (i.e., $1$-independent sets). More generally, the \textbf{$d$-independence complex} $\mathrm{Ind}_d(G)$ has vertices $V(G)$ and faces consisting of all $d$-independent subsets of $V(G)$. While the algebraic and combinatorial properties of $\mathrm{Ind}(G)$ have been well-explored,
%(see, e.g., \cite{FV, HT1, Fahimeh, MV}), 
higher independence complexes $\mathrm{Ind}_d(G)$ for $d \geq 2$ are less understood, with recent advances in \cite{DSS, DS, Roy}. 

Interval graphs, defined as intersection graphs of intervals on a line, form a fundamental class with applications in scheduling, genetics, and beyond \cite{Gardi, G, R2}. A graph $G$ is an \textbf{interval graph} if its vertices can be labeled by nonempty intervals such that two vertices are adjacent precisely when their intervals intersect. A \textbf{proper interval graph} is one where no interval properly contains another, and a \textbf{unit interval graph} requires all intervals to have the same length. Equivalently, as introduced in \cite{Herzog1} to study binomial edge ideals with quadratic Gr\"{o}bner bases, a graph $G$ is \textbf{closed} if there exists a labeling $[n]$ of its vertices such that for edges $ij$ and $ik$ with $i < j < k$, we have $jk \in E(G)$, and for edges $ij$ and $kj$ with $i < k < j$, we have $ik \in E(G)$.  There are generalizations of these graphs for simplicial complexes, such as unit interval, closed, under closed, poor closed, etc. \cite{B+S+V, Fahimeh2} provide hierarchies of such generalizations. In this work, we study generalization of interval graphs, including strong interval, unit interval and under closed, for simplicial complexes (see Definitions \ref{1.1}). We establish hierarchies and equivalences among these classes of simplicial complexes in Lemma \ref{1.2} and Theorem \ref{1.4}.  These relationships lead to conditions for  sortability of $d$-independence complexes (Corollary \ref{B}), constructions of normal Cohen-Macaulay domains (Corollary~\ref{CM}), and characterizations of interval graphs (Theorem~\ref{A}). Moreover, Theorem \ref{Prop} extends the main result of \cite{C+R} to simplicial complexes, showing that closed and proper interval graphs coincide. 

We also derive forbidden subgraph results, including chordality and claw-freeness (Propositions~\ref{1.5} and~\ref{1.7}), culminating in sortability of $\mathrm{Ind}_d(G)$ (Corollary~\ref{2.5}) and characterizations of $d$-unit interval cycles and forests (Corollary~\ref{2.4}). So we extend \cite{C+R}'s result to simplicial complexes and provide new characterizations beyond cycles/forests.

\section{Variant of Interval Simplicial Complexes}
Interval graphs admit natural generalizations to higher-dimensional simplicial complexes, capturing richer combinatorial structures with algebraic implications. We study several such generalizations, including strong interval, unit interval and under closed simplicial complexes, to extend the framework of interval graphs and explore their properties. These definitions, inspired by \cite{B+S+V}, facilitate the study of higher independence complexes and their associated algebras.

\begin{definitions}\label{1.1}
\begin{enumerate}
\item $\Delta$ is called a \textbf{strong interval simplicial complex} if there exists a labelling $[n]$ on $V(\Delta)$ and a family $\{I_i\}_{i\in [n]}$ of intervals such that for each $1\leq i_1<\dots <i_{d+1} \leq n$,  $i_1 \dots  i_{d+1}\in \mathcal{F}(\Delta )$ if and only if $\bigcup _{1\leq k \leq d+1}I_{i_k}$ is an interval. In this case $\{I_i\}_{i\in [n]}$ is called an \textbf{interval representation} for $\Delta$ and for each $i\in [n]$, if $I_i=[a_i, b_i]$, then we set $\mathrm{left}I_i=a_i$ and $\mathrm{right}I_i=b_i$. 
    If all the intervals in an interval representation have unit length, $\Delta$ is called a \textbf{strong unit interval simplicial complex}. If no interval is properly contained in another, $\Delta$ is called a \textbf{strong proper interval simplicial complex}.

\item (\cite[Definition 30]{B+S+V}) $\Delta$ is called a \textbf{unit interval simplicial complex} if there is a labelling $[n]$ on $V(\Delta)$ such that for each facet $F: i_1\dots i_{d+1}$, every integers $j_1, \dots, j_{d+1}$ with $i_1\leq j_1<\dots <j_{d+1}\leq i_{d+1}$ form a facet.
    
\item (\cite[Definition 31]{B+S+V}) $\Delta$ is called a \textbf{under closed simplicial complex} if there is a labelling $[n]$ on $V(\Delta)$ such that for each facet $F: i_1\dots i_{d+1}$, every integers $j_1, \dots, j_{d+1}$ with $j_1=i_1, j_2\leq i_2, j_{d+1}\leq i_{d+1}$ form a facet. 
\end{enumerate}
$G$ is called $d$-strong interval (resp. $d$-strong unit interval, $d$-strong proper interval,  $d$-unit interval, $d$-under closed) if $\Delta_d(G)$ is so.
\end{definitions}

Here we made some observations in this direction:
\begin{remarks}\label{remarks1.2}
\begin{enumerate}
\item In an interval representation, one can consider the intervals in real line. In addition, one may assume that all intervals in an interval representation are closed.

\item $1$-strong interval graphs and $1$-under closed graphs are precisely interval graphs (\cite[Theorem 4]{Ol}). 

\item For a connected graph $G$,  $1$-unit interval is the same as proper interval and closed  (cf. \cite[Proposition 1 and Theorem 1]{LO} and \cite[Corollary 2.5]{C+R}).

\item A graph $G$ is $d$-strong interval (resp. $d$-unit interval, $d$-under closed) if and only if all its connected components are so. The same holds for simplicial complexes.

\item Multiple $d$-complexes may correspond to the same labeling for $d$-under closed or $d$-strong interval graphs (e.g., $C_4$ and $K_4$ are both $3$-under closed and $3$-strong interval graphs with the same interval representation).
\end{enumerate}
\end{remarks}

To understand the hierarchy among these new classes, we first generalize a result from \cite[Theorem 4]{Ol} to higher-dimensional complexes. To this aim firstly we bring an equivalent description for under closed $d$-complexes.

\begin{lemma}\label{global}
A $d$-complex $\Delta$ is under closed if and only if there is a labelling $[n]$ on $V(\Delta)$ such that  for each facet $F: i_1\dots i_{d+1}$ and each $j$ with $i_1\leq j\leq i_{d+1}, j\notin F$, we have $\{i_1, \dots , i_d, j\}\in \mathcal{F}(\Delta)$.
\end{lemma}
\begin{proof}
$(\Longleftarrow )$ Consider a facet $F: i_1\dots i_{d+1}$ and an integer $j$ with $i_1\leq j\leq i_{d+1}, j\notin F$. Assume $\{j_1, \dots , j_{d+1}\}=\{i_1, \dots , i_d, j\}$, where $j_1< \dots < j_{d+1}$. Thus there exists an integer $k$ with $2\leq k \leq d+1$ such that $j_\ell =i_\ell$ for all $1\leq \ell \leq k-1$, $j_k=j$ and $j_\ell=i_{\ell -1}$ for all $k+1\leq \ell \leq d+1$. This shows $j_1=i_1, j_2\leq i_2, \dots , j_{d+1}\leq i_{d+1}$. So $\{i_1, \dots , i_d, j\}\in \mathcal{F}(\Delta)$ by under closedness.

$(\Longrightarrow )$ Consider a facet $F: i_1\dots i_{d+1}$ and integers $j_1, \dots , j_{d+1}$ with $j_1=i_1, j_2\leq i_2, \dots , j_{d+1}\leq i_{d+1}$. Let $s=\min \{\ell \ | \ j_\ell \notin F\}$. Then by assumption, $F'=\{i_1, \dots, i_d, j_s\}\in \mathcal{F}(\Delta )$. By replacing $F$ with $F'$  and iterating this way, we obtain $j_1 \dots j_{d+1}\in \mathcal{F}(\Delta )$ as desired.
\end{proof}

\begin{lemma}\label{1.2}
Every strong interval $d$-complex is under closed.
\end{lemma}
\begin{proof}
Suppose $\{I_i\}_{i\in [n]}$ is an interval representation for $\Delta$. We relabel the vertices of $\Delta$ by $[n]$ in such a way that $i<j$ if either $\mathrm{left}I_i<\mathrm{left}I_j$ or $\mathrm{left}I_i=\mathrm{left}I_j, \mathrm{right}I_i\leq\mathrm{right}I_j$. Now assume that $F:  i_1\dots i_{d+1}\in \mathcal{F}(\Delta )$, $i_1\leq j\leq i_{d+1}$ and  $j\notin F$. Since $F\in \mathcal{F}(\Delta)$, $\bigcup_{1\leq k \leq d+1}I_{i_k}$ is an interval. Also we know that $ \mathrm{left}I_{i_1}\leq \dots \leq \mathrm{left}I_{i_{d+1}}$. So we should have $\mathrm{left}I_{i_{d+1}}\leq \mathrm{right}I_{i_\ell}$ for some $1\leq \ell \leq d$. Since $\mathrm{left}I_j\leq \mathrm{left}I_{i_{d+1}}$, $\mathrm{left}I_{j}\leq \mathrm{right}I_{i_\ell}$. Therefore $(\bigcup_{1\leq k \leq d} I_{i_k})\cup I_j$ is an interval and so $\{i_1, \dots , i_d, j\}\in \mathcal{F}(\Delta )$ as desired. Now Lemma \ref{global} completes the proof.
\end{proof}

We next extend results from \cite{C+R} to higher dimensions. We need the following lemma for this task.

\begin{lemma}\label{lem2.4}
The following statements are equivalent:
\begin{itemize}
  \item[(1)] $\Delta$ is unit interval.
  \item[(2)]  There is a labelling $[n]$ on $V(\Delta)$ such that for each facet $F: i_1\dots i_{d+1}$ and each $j$ with $i_1\leq j\leq i_{d+1}, j\notin F$, $j$ and $i_k$ belong to a facet and $\{i_1, \dots,  \widehat{i_k}, \dots, i_{d+1}, j\}\in \mathcal{F}(\Delta )$  for all $k$ with $1\leq k \leq d+1$. 
  \item[(3)] There is a labelling $[n]$ on $V(\Delta)$ such that for each facet $F: i_1\dots i_{d+1}$ and each $j$ with $i_1\leq j\leq i_{d+1}, j\notin F$, $j$ and $i_k$ belong to a facet for some $k$ with $1\leq k \leq d+1$, and for such integers $k$ we have  $\{i_1, \dots,  \widehat{i_k}, \dots, i_{d+1}, j\}\in \mathcal{F}(\Delta )$.
\end{itemize}
\end{lemma}
\begin{proof}
$(1\Longrightarrow 2)$  and $(2\Longrightarrow 3)$ are immediate.  

$(3 \Longrightarrow 2)$ If $\Delta$ satisfies (3), for any facet $F: i_1 \dots i_{d+1}$ and integer $j$ with $i_1 \leq j \leq i_{d+1}$, $j \notin F$, there exists a facet $H$ containing $j$ and $i_k$ for some $k$ and $\{i_1, \dots, \widehat{i_k}, \dots, i_{d+1}, j\}$ is a facet which contains $j$ and $i_{k'}$ for any $k' \neq k$. 

$(2\Longrightarrow 1)$ Assume $F: i_1\dots i_{d+1}$ is a facet of $\Delta$ and $ i_1\leq j_1< \dots <j_{d+1}\leq i_{d+1}$. Suppose that $s=\min \{\ell \ | \ j_\ell\notin \{i_1, \dots , i_{d+1}\}\}$ and $r=\min \{\ell \  | \ i_\ell \notin \{j_1, \dots , j_{d+1}\}\}$. By~(2) $F'=\{i_1, \dots , \widehat{i_r}, \dots , i_{d+1}, j_s\}\in \mathcal{F}(\Delta )$. By repeating this procedure for $F'$ instead of $F$ we arrive at the conclusion that $\{j_1, \dots , j_{d+1}\}$ is a facet as required.
\end{proof}

\begin{theorem}\label{Prop}
A $d$-dimensional simplicial complex $\Delta$ is unit interval if and only if  Condition $*$ satisfies. 

$*$: There is a labelling $[n]$ on $V(\Delta)$ such that for each facet $F: i_1\dots i_{d+1}$ and each $j$ with $i_1\leq j\leq i_{d+1}, j\notin F$, if there exists a facet containing $j$ and $i_k$ for some $1\leq k \leq d+1$, then  $\{i_1, \dots,  \widehat{i_k}, \dots, i_{d+1}, j\}\in \mathcal{F}(\Delta )$. 

In particular each graph $G$ is closed if and only if it is proper interval.
\end{theorem}
\begin{proof}
$(\Longrightarrow )$ is obvious.

$(\Longleftarrow )$ One can verify it is sufficient to prove the result for connected simplicial complex $\Delta$. Assume $\Delta $ satisfies $*$ while it is not unit interval. According to Lemma \ref{lem2.4}$(1 \Longleftrightarrow 3)$ there exists a facet $F: i_1 \dots  i_{d+1}$ and an integer $j$ with $i_1\leq j\leq i_{d+1}, j\notin F$ such that there is no facet containing $j$ and $i_k$ for all $1\leq k \leq d+1$. We assume that $i_{d+1}$ and $j$ are the largest such integers and look for a contradiction.

Consider a sequence $ r_0, G_1, r_1, \dots , r_{s-1}, G_s, r_s$, with $i_1=r_0$ and $j=r_s$ such that for each $k$ with $1\leq k \leq s$, $r_k$ is a vertex of $\Delta$ and $G_k$ is a facet containing $r_{k-1}$ and $r_k$. Suppose this is a shortest such sequence from $i_1$ to $j$. First of all by descending induction on $k$ we show $r_{k-1}<r_k$  for each $k$ with $1\leq k \leq s$.

For $k=s$ we get to contradiction in each of the following cases and so $r_{s-1}<r_s$.
\begin{itemize}
  \item $r_{s-1}> i_{d+1}$.  Then $r_s=j<i_{d+1}<r_{s-1}$. Since $G_s$ is a facet containing $r_{s-1}, r_s$, if $i_{d+1}$  doesn't belong to any facet together with an element of $G_s$, this contradicts to the choice of $i_{d+1}$ or $j$. So there exists a facet containing $i_{d+1}$ and $j$ since $\Delta$ is proper interval.
  \item $r_{s-1}= i_{d+1}$.  Then $G_s$ is a facet containing $i_{d+1}$ and $j$.
  \item $j<r_{s-1}< i_{d+1}$. Then $i_1<j <r_{s-1}<i_{d+1}$. By the choice of $j$ and unit interval property of $\Delta$, there is a facet, say $H$, containing $i_1, r_{s-1}$. Since $i_1<j<r_{s-1}$, $G_s$ is a facet containing $j,r_{s-1}$ and also $\Delta$ is unit interval interval, $(H\setminus\{r_{s-1}\})\cup \{j\}$ is a facet containing $j$ and $i_1$.
  \item $r_{s-1}=j$. This leads to a shorter sequence from $i_1$ to $j$.
\end{itemize}
Now suppose $1<k<s$ and we have shown $r_{k-1}<\dots < r_s$.  For $r_{k-2}$ in each of the following cases we get to a contradiction and so we should have $r_{k-2}<r_{k-1}$ as desired.
\begin{itemize}
  \item $r_{k-1}<r_k<r_{k-2}$. Then since $r_{k-1}, r_{k-2}\in G_{k-1}, r_k, r_{k-1}\in G_k$ and $\Delta$ is unit interval, $(G_{k-1}\setminus \{r_{k-1}\}) \cup \{r_k\}$ is a facet containing $r_k$ and $r_{k-2}$. This leads to a shorter sequence from $i_1$ to $j$.
  \item $r_{k-1}<r_{k-2}<r_k$. Then since $r_{k-1}, r_k\in G_k, r_{k-1}, r_{k-2}\in G_{k-1}$ and $\Delta$ is unit interval, $(G_k\setminus \{r_{k-1}\}) \cup \{r_{k-2}\}$ is a facet containing $r_k$ and $r_{k-2}$. This leads to a shorter sequence from $i_1$ to $j$.
   \item $r_{k-2}=r_k$ or $r_{k-2}=r_{k-1}$. Then we have a shorter sequence from $i_1$ to $j$.
\end{itemize}
So we have shown
$$i_1=r_0<r_1<\dots <r_{s-1}<r_s=j< i_{d+1}.$$
Therefore $i_1<r_1<i_{d+1}, i_1,i_{d+1}\in F, r_1, i_1\in G_1$ and unit interval property imply $H_1=\{i_2, \dots , i_{d+1}, r_1\}\in \mathcal{F}(\Delta )$. 
Then $r_1<r_2<i_{d+1},  r_1, i_{d+1}\in H_1, r_1,r_2\in G_2$ and unit interval property ensure $H_2=\{i_2, \dots , i_{d+1}, r_2\}\in \mathcal{F}(\Delta )$. By continuing in this way we get to a facet
$H_s=\{i_2, \dots , i_{d+1}, j\}$. This contradiction yields the result.

The last assertion follows from the fact that Condition $*$ is the natural generalization of definition of closed graph. 
\end{proof}

\section{Interval and Proper Interval Graphs}
We now establish equivalent characterizations for interval and proper interval graphs, building on the relationships among our generalized classes.

\begin{theorem}\label{1.4}
\begin{enumerate}
\item Suppose that $G$ is a $d$-strong interval (resp. $d$-strong unit interval, $d$-strong proper interval) graph with no isolated vertex. Then it is $k$-strong interval (resp. $d$-strong unit interval, $d$-strong proper interval) for all $k\geq d$ with the same interval representation.

\item Suppose that $G$ is a $d$-under closed graph. Then it is $k$-under closed for all $k\geq d$ with the same labelling on $V(G)$.

\item Suppose $d>1$ and $G$ is a $d$-unit interval graph. Then it is $k$-unit interval for all $k\geq d$ with the same labelling on $V(G)$.
\end{enumerate}
\end{theorem}
\begin{proof}
It is enough to prove the results for $k=d+1$ when $d+2\geq |V(G)|$.

(1) Suppose $G$ is a $d$-strong interval graph on $[n]$ for which $\{I_i\}_{i\in [n]}$ is an interval representation. Assume that $1\leq i_1< \dots < i_{d+2}\leq n$. We should prove that $G[i_1, \dots , i_{d+2}]$ is connected if and only if $\bigcup_{1\leq \ell \leq d+2}I_{i_\ell}$ is an interval.

$(\Longrightarrow )$ If $G[i_1, \dots , i_{d+2}]$ is connected, then it has at least two distinct noncut vertices, say $i_k$ and $i_{k'}$  (we point out that each vertex in cycles is noncut, also if $G$ has no cycle, it should have at least two leaves which are certainly noncut vertices). Hence $G[i_1, \dots ,\widehat{i_k}, \dots, i_{d+2}]$ and $G[i_1, \dots ,\widehat{i_{k'}}, \dots, i_{d+2}]$ are connected. Since $d\in \mathbb{N}$ and $G$ is $d$-strong interval,  $\bigcup_{1\leq \ell \leq d+2, \ell\neq k}I_{i_\ell}$ and $\bigcup_{1\leq \ell \leq d+2, \ell\neq k'}I_{i_\ell}$ are some intervals with nonempty intersection. Hence, $\bigcup_{1\leq \ell \leq d+2}I_{i_\ell}=(\bigcup_{1\leq \ell \leq d+2, \ell\neq k}I_{i_\ell}) \cup (\bigcup_{1\leq \ell \leq d+2, \ell\neq k'}I_{i_\ell})$ is also an interval.

$(\Longleftarrow )$ If $\bigcup_{1\leq \ell \leq d+2}I_{i_\ell}$ is an interval, then suppose $\mathrm{left} \bigcup_{1\leq \ell \leq d+2}I_{i_\ell}=\mathrm{left}I_{i_k}$ and  $\mathrm{right} \bigcup_{1\leq \ell \leq d+2}I_{i_\ell}=\mathrm{right}I_{i_{k'}}$ for some $1\leq k,k'\leq d+2$. Now if $k=k'$, then since $d+2\geq 3$, $I_{i_k}$ contains $I_{i_s}$ and $I_{i_{s'}}$ for distinct integers $s$ and $s'$ with $s,s'\in \{1, \dots, \widehat{k}, \dots , d+2\}$. Thus $\bigcup_{1\leq \ell \leq d+2, \ell\neq s}I_{i_\ell}$ and $\bigcup_{1\leq \ell \leq d+2, \ell\neq s'}I_{i_\ell}$ are intervals. Also if $k\neq k'$, then either $\bigcup_{1\leq \ell \leq d+2, \ell\neq k}I_{i_\ell}$ is an interval or $I_{i_k}$ contains $I_{i_s}$ for some $1\leq s \leq d+2$ with $s\neq k$. Similarly either $\bigcup_{1\leq \ell \leq d+2, \ell\neq k'}I_{i_\ell}$ is an interval or $I_{i_{k'}}$ contains $I_{i_{s'}}$ for some $1\leq s' \leq d+2$ with $s'\neq k'$. Hence in each case $\bigcup_{1\leq \ell \leq d+2, \ell\neq t}I_{i_\ell}$ and $\bigcup_{1\leq \ell \leq d+2, \ell\neq t'}I_{i_\ell}$ are intervals for some distinct integers $t$ and $t'$ with $1\leq t,t'\leq d+2$. Therefore since  $G$ is $d$-strong interval, the graphs $G[i_1, \dots , \widehat{i_t}, \dots , i_{d+2}]$ and $G[i_1, \dots , \widehat{i_{t'}}, \dots , i_{d+2}]$ are connected. Hence $d+2\geq 3$ implies $G[i_1, \dots, i_{d+2}]$  is connected as desired.

(2) Suppose $G$ is a $d$-under closed graph with labelling $[n]$ on $V(G)$. We use the equivalent definition for under closed $d$ complexes presented in Lemma \ref{global}. Consider a connected subgraph of $G$ with vertices $i_1<\dots < i_{d+2}$ and an integer $j$ with $i_1\leq j\leq i_{d+2}$ and $j\notin \{i_1, \dots , i_{d+2}\}$.  Then for connected graph $G[i_1, \dots , i_{d+2}]$ the following cases occur:

$\bullet$ There exists a noncut vertex, say $i_k$, for some $k$ with $1< k <d+2$.
Therefore $G[i_1, \dots , \widehat{i_k}, \dots , i_{d+2}]$ is connected. Since $G$ is $d$-global interval, $G[i_1, \dots , i_{d+1}]$ and\linebreak $G[i_1, \dots , \widehat{i_k}, \dots , i_{d+1}, j]$ are connected. So $G[i_1, \dots , i_{d+1}, j]$ is also connected.

$\bullet$ $i_k$ is a cut vertex for all $k$ with $1<k < d+2$. Hence $G[i_1, \dots , i_{d+1}]$ and $G[i_2, \dots , i_{d+2}]$ are connected. If $i_1<j<i_{d+1}$,  since $G$ is $d$-global interval, $G[i_1, \dots , i_d, j]$ is connected. Hence  $G[i_1, \dots , i_{d+1}, j]$ is also connected. Else $i_2<j<i_{d+2}$. Therefore $G[i_2, \dots , i_{d+1}, j]$ and so $G[i_1, \dots , i_{d+1},j]$ is connected.

(3) Suppose $G$ is a $d$-unit interval graph. Assume $1 \leq i_1 < \dots < i_{d+2} \leq n$ and $i_1 \leq j_1 < \dots < j_{d+2} \leq i_{d+2}$ such that $G[i_1, \dots, i_{d+2}]$ is connected. We prove that $G[j_1, \dots, j_{d+2}]$ is connected. Since $G$ is $d$-unit interval, for any facet ${i_1, \dots, i_{d+1}} \in \mathcal{F}(\Delta_d(G))$, every ${j_1', \dots, j_{d+1}'}$ with $i_1 \leq j_1' < \dots < j_{d+1}' \leq i_{d+1}$ is a facet. We consider two cases for $G[i_1, \dots, i_{d+2}]$.

$\bullet$ There exists a noncut vertex $i_k$, for some $1 < k < d+2$. Thus, $G[i_1, \dots, \widehat{i_k}, \dots, i_{d+2}]$ is connected. Since $G$ is $d$-unit interval and $i_1 \leq j_1 < \dots < j_{d+2} \leq i_{d+2}$, any $d+1$ vertices $\{j_{s_1}, \dots, j_{s_{d+1}}\} \subseteq \{j_1, \dots, j_{d+2}\}$ with $j_{s_1} < \dots < j_{s_{d+1}}$ form a facet of $\Delta_d(G)$. For $d \geq 1$, take any two subsets $\{j_1, \dots, \widehat{j_\ell}, \dots, j_{d+2}\}$ and $\{j_1, \dots, \widehat{j_{\ell'}}, \dots, j_{d+2}\}$ for $\ell \neq \ell'$. These are connected, share $d \geq 1$ vertices, and their union is $G[j_1, \dots, j_{d+2}]$, which is connected.

$\bullet$ Every $i_k$ for $1 < k < d+2$ is a cut vertex. Hence, $G[i_1, \dots, i_{d+2}]$ is a path graph with endpoints $i_1, i_{d+2}$, so $G[i_1, \dots, i_{d+1}]$ and $G[i_2, \dots, i_{d+2}]$ are connected. If $j_{d+2} \leq i_{d+1}$, then $i_1 \leq j_1 < \dots < j_{d+2} \leq i_{d+1}$. Since $G[i_1, \dots, i_{d+1}]$ is connected, $G[j_1, \dots, j_{d+1}]$  and $G[j_2, \dots, j_{d+2}]$ are connected by the $d$-unit interval property.  Their union $G[j_1, \dots, j_{d+2}]$ is connected, as they share $j_2, \dots, j_{d+1}$ and $d \geq 1$.

If $i_{d+1} < j_{d+2} \leq i_{d+2}$, consider $G[i_2, \dots, i_{d+1}, j_{d+2}]$. Since $i_2 \leq i_2 < \dots < i_{d+1} < j_{d+2} \leq i_{d+2}$ and $G[i_2, \dots, i_{d+2}]$ is connected, $G[i_2, \dots, i_{d+1}, j_{d+2}]$ is connected. If $j_1 \geq i_2$, then $i_2 \leq j_1 < \dots < j_{d+2} \leq i_{d+2}$, so $G[j_1, \dots, j_{d+2}]$ is connected. If $i_1 \leq j_1 < i_2$, consider $G[j_1, i_2, \dots, i_{d+1}]$. Since $i_1 \leq j_1 < i_2 < \dots < i_{d+1} \leq i_{d+1}$ and $G[i_1, \dots, i_{d+1}]$ is connected, $G[j_1, i_2, \dots, i_{d+1}]$ is connected. Thus, $G[j_1, i_2, \dots, i_{d+1}, j_{d+2}]$ is connected, as it is the union of $G[j_1, i_2, \dots, i_{d+1}]$ and $G[i_2, \dots, i_{d+1}, j_{d+2}]$, sharing $i_2, \dots, i_{d+1}$.

Now if  $i_k$ is a noncut vertex of $G[j_1, i_2 , \dots , i_{d+1}, j_{d+2}]$ for some $2\leq k \leq d+1$, then similar to the aforementioned case the result holds. Suppose $G[j_1, i_2 , \dots , i_{d+1}, j_{d+2}]$ is a path graph with end vertices $j_1$ and $j_{d+2}$. If $i_2\leq j_2<\dots <j_{d+1} \leq i_{d+1}$, then since $G$ is $d$-unit interval and $G[j_1, i_2 , \dots , i_{d+1}, j_{d+2}]$ is a path graph, we should have $G[j_1, \dots , j_{d+1}]$ and $G[j_2, \dots , j_{d+2}]$ are connected. Their union $G[j_1, \dots, j_{d+2}]$ is connected, as they share $j_2, \dots, j_{d+1}$ and $d \geq 1$. So either $j_2<i_2$ or $i_{d+1}<j_{d+1}$.

Assume that $i_{d+1}<j_{d+1}$ and $i_k$ is the only neighbour of $j_{d+2}$ in $G[j_1, i_2 , \dots , i_{d+1}, j_{d+2}]$. Now $G[i_2, \dots , i_{d+1}, j_{d+2}]$ is connected, $i_2\leq i_2< \dots <\widehat{i_k} < \dots < i_{d+1}<j_{d+1}<j_{d+2}\leq j_{d+2}$ and $d$-unit interval property of $G$ implies $G[ i_2, \dots , \widehat{i_k}, \dots , i_{d+1},j_{d+1},j_{d+2}]$ is connected.  Therefore  $j_{d+1}j_{d+2}\in E(G)$ and $G[ i_2, \dots , \widehat{i_k}, \dots , i_{d+1},j_{d+1}]$ is connected. On the other hand  $G[j_1, i_2, \dots , \widehat{i_k}, \dots , i_{d+1}]$ is connected. So $G[j_1, i_2, \dots , \widehat{i_k}, \dots , i_{d+1},j_{d+1}]$ is  connected. Thus  we have  $G[j_1, \dots , j_{d+1}]$ is connected, because $G$ is $d$-unit interval. This shows $G[j_1, \dots , j_{d+2}]$ is connected as required. The case $j_2<i_2$ is
symmetric.
\end{proof}

Sortability of monomials and simplicial complexes, as explored in \cite{Ene+Herzog,Fahimeh,St95}, links combinatorial properties to algebraic structures. We use these ideas to characterize $d$-proper interval graphs. 

 \begin{corollary} \label{B}
For each positive integer $d>1$, the following statements are equivalent and for $d=1$ the statements 1, 3  (resp. 2, 4) are equivalent.
\begin{enumerate}
\item $G$ is a $d$-unit interval graph;
\item $G$ is a $k$-unit interval graph for all $k\geq d$ with the same labelling on $V(G)$;
\item $\mathrm{Ind}_d(G)$ is sortable;
\item $\mathrm{Ind}_k(G)$ is sortable for all $k\geq d$ with the same labelling on $V(G)$.
\end{enumerate}
\end{corollary}
\begin{proof}
Notice $\mathrm{Ind}_d(G)$ is the independence complex of $\Delta_d(G)$ and so the result follows from Theorem \ref{1.4} and \cite[Theorem 1.1]{F+M}. For $d=1$, see \cite[Theorem 8]{Fahimeh}.
\end{proof}

Suppose $S=\mathbb{K}[x_1, \dots , x_n]$ and for any finite set $F\subseteq [n]$, associate the monomial $x^F=\prod_{i\in F}x_i$ in $S$.
The following result, which is a consequence of Corollary \ref{B}, 
and \cite[Corollary 13]{Fahimeh}, prepares examples of  normal Cohen-Macaulay rings.

\begin{corollary}\label{CM}
Suppose that $G$ is a $d$-unit interval graph for some $d>1$. Then for all integers $k$ and $t$ with $k\geq d, 0\leq t\leq \dim (\mathrm{Ind}_k(G))$ the $\mathbb{K}$-algebra
$$\mathbb{K}[x^F : F\in \mathrm{Ind}_k(G) , |F|=t+1],$$
is Koszul and a normal Cohen-Macaulay domain.
\end{corollary}

We provide a characterization of interval graphs, unifying various perspectives.

\begin{theorem}\label{A}
For a simple graph $G$ the following statements are equivalent:
\begin{enumerate}
\item $G$ is an interval graph.

\item $G$ is a $k$-strong interval graph for all $k\in \mathbb{N}$ with the same interval representation.

\item  $G$ is a $k$-under closed graph for all $k\in \mathbb{N}$ with the same labelling on $V(G)$.

\item $G$ is a $1$-under closed graph.
\end{enumerate}
\end{theorem}
\begin{proof}
$(1\Longrightarrow 2)$ follows from Remark \ref{remarks1.2}(2) and  Theorem \ref{1.4}(1).

$(2\Longrightarrow 3)$ follows from Lemma \ref{1.2}.

$(3\Longrightarrow 4)$ It is enough to set $k=1$.

$(4\Longrightarrow 1)$ Although in view of \cite[Theorem 4]{Ol} the result holds, but here we rewrite the proof of the implication  $(ii \Longrightarrow i)$  of \cite[Theorem 4]{Ol} without direct using of \cite[Theorem 2]{GH} as follows:

Firstly fix a labelling $[n]$ on $V(G)$ with the property as in Definition \ref{1.1}(2) for $d=1$. Suppose that $C_1, \dots, C_m$ are all maximal cliques in $G$. If for each $v\in V(G)$ we set $I_v=\{C_i \ | \ v\in C_i\}$, then clearly for all $u,v\in V(G)$, $uv\in E(G)$ if and only if $I_u\cap I_v\neq \emptyset$. So it is enough to prove that $I_v$ is an interval. To this aim set $C_i\prec C_j$ if either $\min C_i <\min C_j$ or $\min C_i=\min C_j, \max C_i<\max C_j$. Hint that two distinct maximal cliques can not have the same minimum and maximum. Because suppose that $C_a$ and $C_b$ are two distinct maximal cliques with the vertices $j_1<\dots <j_a$ and $j'_1< \dots <j'_b$ respectively such that $j_1=j'_1, j_a=j'_b$. Assume that
$$t=\max\{\ell \ | \ 1\leq \ell \leq a, j_\ell\not\in C_b\}, s=\max\{\ell \ | \ 1\leq \ell \leq b, j'_s\not\in C_a\}.$$
 Without loss of generality we may assume that $j_t<j'_s$. Then for each $j_\ell<j'_s$, since $j_\ell j_a\in E(G)$ and $j_\ell<j'_s<j'_b=j_a$, we should have $j_\ell j'_s\in E(G)$. Also for each $j_\ell>j'_s$ since $j_t<j_\ell$ and $C_b$ is a clique, we have $j_\ell j'_s\in E(G)$. Hence adding $j'_s$ to $C_a$ forms a larger clique which is a contradiction. Notice also that if $C_i\prec C_j$, then $\min C_i \leq \min C_j$ and $\max C_i\leq \max C_j$. Because if $\max C_i>\max C_j$, since $\min C_i$ is adjacent to $\max C_i$, it is also adjacent to each vertex in $C_j$ and maximality of $C_j$ ensures that $\min C_i=\min C_j$ which yields to $\max C_i <\max C_j$ which is a contradiction.

Now to show that $I_v$ is an interval assume in contrary that $C_1\prec C_2 \prec C_3$  are three maximal cliques with the vertices $i_{1,1}<\dots < i_{1,k_1}$ and $i_{2,1}<\dots < i_{2,k_2}$ and $i_{3,1}<\dots < i_{3,k_3}$ respectively such that $v$ belongs to $C_1$ and $C_3$ but it doesn't belong to $C_2$. Then we may assume that $i_{2,r}$ is the minimum vertex in $C_2$ which is not adjacent to $v$ for some $1\leq r \leq k_2$. Now the following cases happen all of which lead to contradictions.
\begin{itemize}
\item $i_{2,r}<v<i_{2,r'}$ for some $r'$. Since $i_{2,r}i_{2,r'}\in E(G)$, we should have $i_{2,r}v\in E(G)$.

\item $i_{2,k_2}<v$. Since $C_1\prec C_2$ and $v\in C_1$, we should have $i_{1,k_1}\leq i_{2,k_2}<v\leq i_{1,k_1}$ and so $v=i_{2,k_2}$.

\item $v<i_{2,r}$. Since $C_2\prec C_3$ and $v\in C_3$, we have $v<i_{2,r}\leq i_{2,k_2}\leq i_{3,k_3}$ and $vi_{3,k_3}\in E(G)$. Thus $vi_{2,r}\in E(G)$.
\end{itemize}
\hspace{1cm}
\end{proof}

\section{Some Properties of Variant of Interval Simplicial Complexes}
We now investigate structural properties of our generalized classes, focusing on cycle lengths and forbidden subgraphs.
Recall that a $d$-complex is called chordal if there is a labelling $[n]$ on $V(\Delta )$ such that for any two facets $F$ and $G$ with vertices $i_1< \dots < i_{d+1}$ and $j_1< \dots < j_{d+1}$ with $i_{d+1}=j_{d+1}$, the complex $\Delta$ contains the full $d$-skeleton on $F\cup G$. According to Remark 35 in \cite{B+S+V}, all unit interval simplicial complexes are chordal. So if $G$ is $d$-unit interval, then $\Delta_d(G)$ is chordal and thus it can be checked that $G$ doesn't have any induced cycle with length $\geq d+3$. Moreover recall that a graph is called \textbf{chordal (triangulated)}  if every cycle of length strictly greater than $3$ possesses a chord. Remark that interval graphs are chordal and also proper interval graphs are chordal and clawfree (cf. \cite{BW}, \cite{GH}, \cite{Herzog1} and \cite{R1}). The next two propositions generalize these facts. To this aim firstly we need to generalize the concept of claw.

\begin{definition}
Suppose that $G_1, G_2$ and $G_3$ are distinct connected graphs with at least two and at most $d+1$ vertices. Assume for each $i,j$ with $1\leq i<j\leq 3$, $|V(G_i)\cup V(G_j)| \geq d+1$ and $|V(G_1)\cap V(G_2) \cap V(G_3)|=1$. Then $G_1\cup G_2 \cup G_3$ is called a \textbf{$d$-claw} if each path between these three graphs just passes through their common vertex. Moreover any connected graph with $d+2$ vertices which has only three distinct leaves is called a \textbf{$d$-paw}. A graph is called \textbf{$d$-clawfree} (resp. \textbf{$d$-pawfree}) if it doesn't have any $d$-claw (resp. $d$-paw) as an induced subgraph.
\end{definition}
Note that $1$-claw and $2$-paw are exactly $K_{1,3}$ that is a claw graph. Also note that there is not any $1$-paw. 
\begin{proposition}\label{1.5}
Assume that $G$ is a $d$-under closed (resp. $d$-strong interval, $d$-unit interval) graph. Then $G$ doesn't have any induced cycle with length $\geq d+3$. In particular any interval, proper interval or closed graph is chordal.
\end{proposition}
\begin{proof}
In view of Lemma \ref{1.2}, Remark \ref{remarks1.2}(3) we just prove the assertion when $G$ is $d$-under closed. By the way of contradiction assume there is an induced cycle $C$ with vertices $i_1<\dots <i_k$ for some $k\geq d+3$. Suppose that $C=(j_1, \dots, j_k)$ such that $j_{d+1}=i_{k}$. Let $j_s=\min \{j_1, \dots , j_{d+1}\}$. If $j_s<j_\ell <j_{d+1}$ for some $d+2\leq \ell \leq k-1$, then since $G[j_1, \dots , j_{d+1}]$ is connected, $G[j_1, \dots , j_d, j_\ell]$ should also be connected, a contradiction. Else $j_\ell <j_s$ for all $\ell$ with $d+2\leq \ell \leq k-1$. So since $G[j_{s+1}, \dots , j_{s+d+1}]$ is connected (set $j_r=j_{r '}$ when $r\equiv r' (\text{mod} \ k)$ for some $1\leq r' \leq k$) and $j_{d+2}<j_s < j_{d+1}$, the subgraph $G[j_s, \dots , \widehat{j_{d+1}}, \dots , j_{s+d+1}]$ of $G$ should be connected. This is a contradiction too.
\end{proof}

\begin{proposition}\label{1.7}
Any $d$-unit interval graph is $d$-clawfree and $d$-pawfree.
\end{proposition}
\begin{proof}
Assume $G$ is a $d$-unit interval graph with labelling $[n]$ on $V(G)$. If $G$ has an induced $d$-claw, then there exist connected subgraphs  $G_1, G_2$ and $G_3$ with  $i$ as the only common vertex and distinct vertices $j_k\in V(G_k)$ with $j_k\neq i$ when $k=1,2,3$. Hence $j_r$ lies between $i$ and $j_s$ for some $1\leq r,s \leq 3$. Without loss of generality suppose $i< j_1< j_2$. Now assume $H\in \mathcal{F}(\Delta _d(G))$ contained in $G_2\cup G_3$ containing $i$ and $j_2$.  Since $G$ satisfies Condition $*$ by Theorem \ref{Prop} and possesses a facet of $\Delta_d(G)$ containing $i$ and $j_1$,  $(H\setminus \{i\})\cup \{j_1\}$ must be a facet of $\Delta_d(G)$ which is a contradiction.

If $G$ has an induced $d$-paw, say $P$, with three distinct leaves $i, j$ and $k$. Without loss of generality suppose that $i<j<k$ and $\ell$ is the joint vertex of $j$. Since $i<j<k$, $V(P)\setminus \{j\}$ is a facet of $\Delta_d(G)$ containing $i,k, \ell$ , there is a facet in $\Delta_d(G)$ (say $V(P)\setminus \{i\}$) containing $j$ and $\ell$ and $G$ satisfies $*$, $V(P)\setminus \{\ell\}$ must be connected which is a contradiction.
\end{proof}

Combining Corollary \ref{B} and Propositions \ref{1.5} and \ref{1.7} gives the following corollary. 
Hence the assumption of being a tree in \cite[Corollary 9]{Fahimeh} is redundant.

\begin{corollary}\label{2.5}
If $\mathrm{Ind}_d(G)$ is sortable for some $d>1$, then $G$ is $k$-clawfree and $k$-pawfree for all $k\geq d$ and it doesn't have any induced cycle with length $\geq d+3$. Also if $\mathrm{Ind}(G)$ is sortable, then $G$ is clawfree and chordal.
\end{corollary}

The corona construction generalizes adding whiskers or complete graphs to a graph \cite{F+H, V}. Suppose $\mathcal{H}$ is a family of simple graphs indexed by vertices of $G$ with disjoint sets of vertices. The \textbf{corona} $Go\mathcal{H}$ of $G$ and $\mathcal{H}$ is the disjoint union of $G$ and $H_x$s with additional edges joining each vertex $x$ of $G$ to all vertices of $H_x$. The following corollary can be immediately gained from Proposition \ref{1.7}.

\begin{corollary}
Assume that $G$ is a graph which has a connected subgraph $G'$ with $d-1$ vertices and $\mathcal{H}$ is a family of simple graphs $H_x$ indexed by $V(G)$ with disjoint sets of vertices such that $\bigcup_{x\in V(G')} V(H_x)$ contains a $1$-independent subset with three elements. Then $Go\mathcal{H}$ is not $d$-unit interval.
\end{corollary}

Finally, we apply our results to specific graph classes, providing explicit characterizations.

\begin{corollary}\label{2.4}
\begin{itemize}
\item[a.] If $G=C_n$ with $n\geq 3$, then the followings are equivalent for $d>1$ and a labelling $[n]$ on $V(G)$. When $d=1$, the odd numbers (resp. the even numbers) are equivalent.
\begin{enumerate}
\item $\mathrm{Ind}_d(G)$ is sortable.
\item $\mathrm{Ind}_k(G)$ is sortable for all $k\geq d$.
\item $G$ is $d$-unit interval.
\item $G$ is $k$-unit interval for all $k\geq d$.
\item$G$ is $d$-under closed.
\item $G$ is $k$-under closed for all $k\geq d$.
\item $d\geq n-2$.
\end{enumerate}

\item[b.] If $G$ is a forest, then the following statements are equivalent for $d>1$ and a labelling $[n]$ on $V(G)$. When $d=1$, the odd numbers (resp. the even numbers) are equivalent.
\begin{enumerate}
\item $\mathrm{Ind}_d(G)$ is sortable.
\item $\mathrm{Ind}_k(G)$ is sortable for all $k\geq d$.
\item $G$ is $d$-unit interval.
\item $G$ is $k$-unit interval for all $k\geq d$.
\item $G$ is disjoint union of path graphs or trees with at most $d+1$ vertices.
\end{enumerate}
\end{itemize}

\end{corollary}
\begin{proof}
a. $(1 \Longleftrightarrow 2 \Longleftrightarrow 3 \Longleftrightarrow 4)$  Corollary \ref{B}.

$(4 \Longrightarrow 6)$ follows from \cite[Theorem 50]{B+S+V}.

$(6 \Longrightarrow 5)$ is clear.

$(5 \Longrightarrow 7)$ If $(7)$ is violated, then $d+3\leq n$. Thus by Proposition \ref{1.5}, $G$ doesn't have any induced cycle of length $n$ which is a contradiction.

$(7 \Longrightarrow 3)$ By standard labelling $[n]$ on $V(G)$ such that $G=(1, \dots , n)$ one can easily check that $G$ is $d$-unit interval.

b. $(1 \Longleftrightarrow 2 \Longleftrightarrow 3 \Longleftrightarrow 4)$ Corollary \ref{B}.

$(3 \Longrightarrow 5)$ follows from Proposition \ref{1.7}.

$(5 \Longrightarrow 3)$ By standard labelling $[n]$ on $V(P_n)$ such that $E(P_n)=\{ij \ | \ 1\leq i \leq n-1, j=i+1\}$, one can easily check that $P_n$ is $d$-unit interval. Also obviously any graph with at most $d+1$ vertices is $d$-unit interval. Hence the result holds by Remark \ref{remarks1.2}(5).
\end{proof}

\begin{example}
We demonstrate that $d$-unit interval graphs encompass structures beyond cycles and forests. Consider a graph $G$ on vertex set $[4] = \{1, 2, 3, 4\}$ with edge set $E(G) = \{12, 23, 24, 34\}$.  Then by the current labelling, $\Delta_2(G)=\langle 123, 124, 234\rangle$ satisfies Definition 2.1(2). Thus, $\Delta_2(G)$ is 2-unit interval, confirming $G$ is a 2-unit interval graph. The 2-independence complex $\mathrm{Ind}_2(G) = \langle 12, 23, 24, 134 \rangle$ comprises 2-independent sets. By Corollary~\ref{B}, the 2-unit interval property of $G$ implies that $\mathrm{Ind}_k(G)$ is sortable and $G$ is $k$-unit interval for all $k\geq 2$.

Additionally, by assigning intervals $[0, 1]$, $[1, 2]$, $[2, 3]$, $[2, 3]$ to vertices $1, 2, 3, 4$, respectively, $\Delta_2(G)$ is 2-strong unit interval.  By Theorem~\ref{1.4}(1), which likely extends strong unit interval properties to higher $k$, $G$ is $k$-strong unit interval for $k \geq 2$, provided $\Delta_k(G)$ maintains this property. 
 \end{example}

%\textbf{Acknowledgments:}The authors would like to thank the referee for his/her valuable comments which substantially improved the quality of the paper.

%\providecommand{\bysame}{\leavevmode\hbox
%to3em{\hrulefill}\thinspace}

%\newpage

%\textbf{Statements and Declarations.}

%The author declares that no funds, grants, or other support were received during the preparation of this manuscript. Also the author has no relevant financial or non-financial interests to disclose.

\end{document}